\documentclass[10pt,a4paper,leqno]{amsart}
\usepackage{amssymb, amsmath, amscd}
\usepackage[all]{xy}
\usepackage[dvipsnames]{xcolor}
\usepackage[hypertexnames=false]{hyperref}

\usepackage{color}

\DeclareMathOperator{\h}{\rm hes}
\DeclareMathOperator{\diver}{\rm div}
\DeclareMathOperator{\Hes}{\rm Hes}
\DeclareMathOperator{\Ric}{\rm Ric}

\DeclareMathOperator{\spanned}{\rm span}

\newtheorem{theorem}{Theorem}[section]
\newtheorem{lemma}[theorem]{Lemma}

\newtheorem{corollary}[theorem]{Corollary}

\theoremstyle{definition}

\theoremstyle{remark}
\newtheorem{remark}[theorem]{Remark}

\usepackage{soul}

\begin{document}
\title{Isotropic quasi-Einstein manifolds}
\author[Brozos-V\'azquez, Garc\'ia-R\'io, Valle-Regueiro]{M. Brozos-V\'azquez, E. Garc\'ia-R\'io, 
X. Valle-Regueiro}
\address{MBV: Universidade da Coru\~na, Differential Geometry and its Applications Research Group, Escola Polit\'ecnica Superior, 15403 Ferrol,  Spain}
\email{miguel.brozos.vazquez@udc.gal}
\address{EGR-XVR: Faculty of Mathematics,
University of Santiago de Compostela,
15782 Santiago de Compostela, Spain}
\email{eduardo.garcia.rio@usc.es $\,\,$ xabier.valle@usc.es}
\thanks{Supported by project MTM2016-75897-P (Spain).}
\date{}

\maketitle

\begin{abstract} 
We investigate the local structure of four-dimensional Lo\-rent\-zian quasi-Einstein manifolds under conditions on the Weyl tensor. We show that if the Weyl tensor is harmonic and the potential function preserves this harmonicity then, in the isotropic case, the manifold is necessarily a $pp$-wave. Using the quasi-Einstein equation, further conclusions are obtained for $pp$-waves.
\end{abstract}

\section{Introduction}
Let $(M,g)$ be a Lorentzian manifold of dimension $4$. $(M,g)$ is said to be \emph{quasi-Einstein (qE)} if there exist a smooth function $f$ and a constant $\mu$ so that the Bakry-\'Emery-Ricci tensor $\rho_f^\mu:=\rho+\text{Hes}_f-\mu df\otimes df$ is a multiple of the metric $g$:
\begin{equation}\label{eq:general quasi-Einstein}
{\Hes}_f+\rho-\mu \, df\otimes df=\lambda\, g.
\end{equation}
Here $\rho$ and $\operatorname{Hes}_f$ denote the Ricci tensor and the Hessian of $f$, respectively. The Bakry-\'Emery-Ricci tensor $\rho_f^\mu$ naturally appears on manifolds with density and was recently used to extend splitting theorems (see \cite{Woolgar, WW}) or to obtain singularity theorems of cosmological type (see \cite{GW, KWW, RW, Woolgar}).For convenience, we denote this structure with the quadruple $(M,g,f,\mu)$. $\lambda$ is a function determined by the trace of \eqref{eq:general quasi-Einstein}: 
\begin{equation}\label{eq:trace-qE}
4\lambda=\Delta f+\tau-\mu g(\nabla f,\nabla f),
\end{equation}  
where $\tau$ denotes the scalar curvature.
Quasi-Einstein structures generalize well-known families of manifolds such as \emph{Einstein manifolds}, \emph{conformally Einstein manifolds}, \emph{gradient Ricci solitons} or \emph{$\kappa$-Einstein solitons} \cite{Sandra-JGeomAnal,cao-tran,CCDMM,Friedan,KNT}. 


If $\mu\neq 0$, equation~\eqref{eq:general quasi-Einstein} is linearized by the change of variable $h=e^{-\mu f}$ and transforms into $\Hes_h -\mu h\rho=-\mu h \lambda g$. Particularizing $\mu=1$ and $\lambda=-\frac14\left(\frac{\Delta h}{h}-\tau\right)$ one obtains the {\it static perfect fluid equation} \cite{Kobayashi1982}, where $h$ is an arbitrary function. Moreover, one recovers the characterizing equation of critical metrics for the quadratic functional given by the $L^2$-norm of the scalar curvature on metrics of fixed volume by additionally specifying $h=\tau$ (see \cite{Berger}).

The Einstein equation on a general warped product structure gives rise to the qE equation on the base for constant $\lambda$. Indeed, for a warped product $M=B\times_\varphi F$, if $M$ is Einstein then $(B,g_B)$ is qE for $\mu=\frac{1}{\dim F}>0$. Furthermore, the converse is also true for a suitable fiber $F$ and one can build examples of Einstein warped products from solutions to the qE equation \cite{Kim-Kim}.



The potential function of \eqref{eq:general quasi-Einstein} defines a conformal deformation
$\tilde g=e^{-f}g$ that  transforms the Ricci tensor as follows (see, for example, \cite{Kuhnel-Rademacher})
\[
\tilde \rho =\rho+\operatorname{Hes}_f+\frac{1}{2}\,df\otimes df+\frac{1}{2}\,(\Delta f-\|\nabla f\|^2)g.
\]
Hence, $\tilde g$ is Einstein if and only if $f$ is a solution to the qE equation for  $\mu=-\frac{1}2$. For this reason, this particular value of $\mu$ is distinguished and solutions to this particular case exhibit a different behavior than solutions for other values of $\mu$ (see \cite{LCFLorentzianQE,Catino} for examples of this fact). For other values of $\mu$, $\mu\neq -\frac12$, a solution to the qE equation gives rise to a conformal metric $\tilde g$ with associated Ricci tensor $\tilde \rho$ given by 
\[
\tilde \rho =\left(\mu+\frac12\right) \, df\otimes df+ \frac14\left(\tau+3\Delta f-\left(\mu+2\right)\|\nabla f\|^2\right)e^{f}\tilde g,
\]
Thus, if $\nabla f$ is timelike, the underlying geometric structure of the manifold is that of a perfect fluid spacetime \cite{Oneill}. These manifolds have been largely investigated; we refer to \cite{shepley, mantica2016} for old and recent examples where conditions on the Weyl tensor were considered.

%

It is well-known that the curvature tensor of a pseudo-Riemannian manifold is determined by its Ricci and Weyl tensors. The qE equation provides information on the Ricci tensor very directly, but not on the Weyl tensor. Hence it is common to find examples in the literature where qE manifolds, or any of its subfamilies, are considered under the hypothesis of local conformal flatness (see, for example, \cite{LCFLorentzianQE,catino2013}). We consider weaker conditions in this regard, namely that the Weyl tensor is harmonic and that the conformal metric $\tilde g=e^{-f} g$ also has harmonic Weyl tensor, i.e. $\operatorname{div} W=\operatorname{div} \tilde W=0$.
 
The divergence of the Weyl tensor is modified by a conformal change $\tilde g=e^{-f} g$ as $\operatorname{div} \tilde W=\operatorname{div} W-\frac12 W(\cdot,\cdot,\cdot,\nabla f)$, so the fact that $\tilde W$ is also harmonic translates into the condition $W(\cdot,\cdot,\cdot,\nabla f)=0$ (see \cite{Kuhnel-Rademacher}). 
Note that the qE equation also provides information on the level sets of the potential function by means of their second fundamental form ${\Hes}_f$.
The situation is very different depending on the character of $\nabla f$: if $\nabla f$ is spacelike or timelike ({\it non-isotropic} case) then the level sets of $f$ are non-degenerate hypersurfaces. We will see that the classification result resembles the Riemannian one in this case, implying that if $\nabla f$ is timelike then the manifold is a Robertson-Walker spacetime (see Subsection~\ref{subsection:non-isotropic}). Hence we are concentrating on the {\it isotropic} case, i.e.  $\nabla f$ is lightlike, and thus the level sets of $f$ are degenerate hypersurfaces. This exhibits the genuinely new and more interesting situations and is treated in  Subsection~\ref{subsection:isotropic}. We show that $4$-dimensional isotropic qE manifolds are $pp$-waves, but not necessarily plane waves (Theorem~\ref{th:isotropic}). We provide  remarks extending the four-dimensional results to higher dimensions. In Section 3 we analyze isotropic qE $pp$-waves more deeply and provide a convenient characterization, showing that either they are locally conformally flat or the necessary and sufficient condition for a $pp$-wave to be isotropic qE is to have harmonic Weyl tensor (Theorem~\ref{th:isotropic-pp-wave}).

\section{Quasi-Einstein Lorentzian manifolds}

The objective in this section is to study the rigidity of the underlying structure of a qE manifold.
In order to analyze the geometric objects associated to the qE equation, we first derive some formulas which involve different tensors giving information of the geometry of the manifold. We fix the curvature sign convention defining $R(X,Y)=\nabla_{[X,Y]}-[\nabla_X,\nabla_Y]$. Recall the expression for the Weyl tensor in dimension $4$: 
\begin{equation}\label{eq:weyl}\small
\begin{array}{l}
W(X,Y,Z,T)=R(X,Y,Z,T)
+\frac{\tau}{6}\{g(X,Z)g(Y,T)-g(X,T)g(Y,Z))\}\\
\noalign{\medskip}
\quad+\frac{1}{2}\{\rho(X,T)g(Y,Z)-\rho(X,Z)g(Y,T)+\rho(Y,Z)g(X,T)-\rho(Y,T)g(X,Z)\}.
\end{array}
\end{equation}

\begin{lemma}\label{lemma:formulae}
Let $(M,g,f,\mu)$ be qE with $\dim M=4$. Then
	\begin{align}
	&2\nabla \Delta f+2\Ric(\nabla f)+\nabla \tau-2\mu\left(\h_f(\nabla f)+\Delta f \nabla f\right)=2\nabla \lambda,\label{eq:5}\\
	\bigskip
	&\nabla \tau+2\mu(3\lambda-\tau) \nabla f+ 2 (\mu-1)\Ric(\nabla f)=6\nabla \lambda,\label{eq:6}\\
	\bigskip
	&R(X,Y,Z, \nabla f)
	=d\lambda(X)g(Y,Z)- d\lambda(Y)g(X,Z)+(\nabla_Y \rho)(X,Z)\label{eq:7} \\
	\smallskip	
	&\qquad-(\nabla_X \rho)(Y,Z)
	+\mu\left\{df(Y)\Hes_f(X,Z)-df(X)\Hes_f(Y,Z)\right\},\nonumber
	\end{align}
\vspace{-0.5cm}
\begin{equation}\label{eq:weylandcotton}
\begin{array}{l}
\!\!\!\!\!\!\! W(X,Y,Z,\nabla f)=2\operatorname{div}W(X,Y,Z)+ \frac{\tau(2\mu+1)}{6}\{ df(Y)g(X,Z)-df(X)g(Y,Z)\}\\
\noalign{\medskip}
\qquad\qquad\qquad\qquad\qquad\qquad +\frac{2\mu+1}{6}\{\rho(X,\nabla f)g(Y,Z)-\rho(Y,\nabla f)g(X,Z)\}\\
\noalign{\medskip}
\qquad\qquad\qquad\qquad\qquad\qquad 
+\frac{2\mu+1}{2}\{\rho(Y,Z)df(X)-\rho(X,Z)df(Y)\}.
\end{array}
\end{equation}	
	
\end{lemma}
\begin{proof}
The divergence of Equation~\eqref{eq:trace-qE} combined with the Bochner formula, $\diver \Hes_f= d\Delta f + \rho(\nabla f,\cdot)$, and
	 the  Contracted Second Bianchi Identity,	$\diver \rho = \frac{1}{2} d\tau$, gives Equation \eqref{eq:5}.
	 Substituting $\Delta f$ and $\h_f(\nabla f)$ using \eqref{eq:general quasi-Einstein}, leads to \eqref{eq:6}.
	Finally, applying \eqref{eq:general quasi-Einstein} one gets:
 \begin{equation*}
\begin{array}{lll}
		(\nabla_X \rho)(Y,Z)& = & \left(\nabla_X \left(\lambda g-\Hes_f+\mu df\otimes df \right)\right) (Y,Z)\\
		\noalign{\medskip}
		& = & d\lambda(X)g(Y,Z)+ \mu \{\Hes_f(X,Y)df(Z)+df(Y)\Hes_f(X,Z)\}\\
		\noalign{\medskip}
		& &+g(\nabla_{\nabla_X Y}\nabla f,Z)-g(\nabla_{X }\nabla_Y \nabla f,Z). 	
	\end{array}
\end{equation*}	
	Therefore,
\begin{align*}
			(\nabla_X \rho)(Y,Z)-&(\nabla_Y \rho)(X,Z)= d\lambda(X)g(Y,Z) -d\lambda(Y)g(X,Z)\\
            &+\mu\{df(Y)\Hes_f(X,Z)- df(X)\Hes_f(Y,Z)\}+R(X,Y,\nabla f,Z),    
	\end{align*}
	from where \eqref{eq:7} follows. The expression for the divergence of the Weyl tensor is
\[
\operatorname{div} W (X,Y,Z)=-\frac12 (\nabla_X \rho)(Y,Z)-(\nabla_Y \rho)(X,Z)
+\frac{1}{12}(X(\tau) g(Y,Z)-Y(\tau) g(X,Z)),
\] 
so one can substitute the curvature term $R(X,Y,Z,\nabla f)$ in the definition of the Weyl tensor \eqref{eq:weyl} using expression \eqref{eq:7}, and then substitute $(\nabla_X \rho)(Y,Z)-(\nabla_Y \rho)(X,Z)$ using the divergence of the Weyl tensor to get Equation~\eqref{eq:weylandcotton}.
\end{proof}

If $(M,g)$ is qE for $\mu=-\frac12$, then $\tilde g=e^{-f}g$ is Einstein, so $\operatorname{div} \tilde W=0$. This implies that $\operatorname{div} W=0$ and $W(\cdot,\cdot,\cdot,\nabla f)=0$ are equivalent conditions. Our techniques do not apply in this particular case and the statements of  Theorem~\ref{Th:non-isotropic} and Theorem~\ref{th:isotropic} below are no longer true if $\mu=-\frac12$. In fact, a locally conformally flat Lorentzian $4$-dimensional manifold is qE for $\mu=-\frac{1}{2}$ and satisfies $\operatorname{div}W=0$ and $W(\cdot,\cdot, \cdot,\nabla f)=0$, but it is not necessarily a warped product or a pp-wave.

\subsection{Non-isotropic quasi-Einstein manifolds}\label{subsection:non-isotropic}
In this subsection we assume that the level sets of the potential function are nondegenerate hypersurfaces and explore the structure of qE manifolds assuming conditions on the Weyl tensor. Since $\|\nabla f\|\neq 0$ we set $E_1={\nabla f}/{\|\nabla f\|}$ and complete it to an orthonormal frame $\{E_1,\dots,E_4\}$ with $\varepsilon_i=g(E_i,E_i)$.

\begin{lemma}\label{lemma:W=0,nabla f eigenvector}
	Let $(M,g,f,\mu)$ be non-isotropic qE with $\mu\neq -\frac{1}{2}$, with $\operatorname{div} W=0$ and $W(\cdot,\nabla f,\cdot,\nabla f)=0$. Then, the Ricci operator diagonalizes in the basis $\{E_1,\dots,E_4\}$.
\end{lemma}

\begin{proof}
	Since $W(X,\nabla f,Z,\nabla f)=0$ for all vector fields $X$, $Z$, since $\operatorname{div} W=0$ and since $\mu\neq -\frac{1}{2}$, we obtain from \eqref{eq:weylandcotton} that
	\begin{equation}\label{eq:lemma W=0 mu neq}
	\begin{array}{rcl}
	0&=&\tau\{df(X)df(Z)-\|\nabla f\|^2 g(X,Z)\}\\
	\noalign{\medskip}
	& &+3\{ \rho(X,Z)\|\nabla f\|^2-\rho(\nabla f,Z)df(X)\}\\
	\noalign{\medskip}
	&&+  g(X,Z)\rho(\nabla f,\nabla f)
	-df(Z)\rho(X,\nabla f).
	\end{array}
	\end{equation}
	We set $X=E_i$, for $i\neq 1$, and $Z=\nabla f$ to see that $2\|\nabla f\|^2 \rho(E_i,\nabla f)=0$, so $\rho(E_i,E_1)=0$. By setting $X=E_i$ and $Z=E_j$ with $i\neq j$ and $i,j\geq 2$, we see that $\rho(E_i,E_j)=0$. 
\end{proof}

The following result shows that a qE manifold whose Weyl tensor is harmonic for the metrics $g$ and $e^{-f}g$ is a Robertson-Walker spacetime if $\nabla f$ is timelike ($\mu\neq -\frac{1}{2}$).

\begin{theorem}\label{Th:non-isotropic}
	Let $(M,g,f,\mu)$ be a non-isotropic qE Lorentzian structure of dimension $4$ with $\mu\neq -\frac{1}{2}$. If $(M,g)$ has harmonic Weyl tensor and $W(\cdot,\cdot, \cdot,\nabla f)=0$, then $(M,g)$ is locally conformally flat and locally isometric to a warped product of the form $I\times_\varphi N$, where $N$ has constant sectional curvature.
\end{theorem}

\begin{proof}
	We use Equation~\eqref{eq:lemma W=0 mu neq} setting $X=Z=E_i$, with $i\neq 1$, and $Y=E_1$, to see that  $3\varepsilon_i \rho(E_i,E_i)=\tau-\varepsilon_1 \rho(E_1,E_1)$. Then, from Equation~\eqref{eq:general quasi-Einstein}, we get that $\Hes_f(E_i,E_j)=0$, for $i\neq j$, and, moreover,
	\[
	\Hes_f(E_i,E_i)=\varepsilon_i\lambda-\rho(E_i,E_i)=\varepsilon_i\left(\lambda-\varepsilon_i \rho(E_i,E_i)\right)=\left(\lambda-\frac{\tau-\varepsilon_1\rho(E_1,E_1)}{3}\right)\varepsilon_i,
	\]
	for all $i\neq 1$. Thus, it follows that the {\color{black}level hypersurfaces} of $M$ are totally umbilical. 
	Furthermore, the distribution generated by $\nabla f$  is totally geodesic.
	Hence $(M,g)$ decomposes locally as a twisted product of the form
	$I\times_\varphi N$, where $I\subset \mathbb{R}$ is an open interval, $N$ is an $(n-1)$-dimensional space and $\varphi$ is a function on $I\times N$ (see \cite{ponge-reckziegel}).
	Moreover, since the Ricci tensor is diagonal, the twisted product reduces to a warped product of the form $I\times_{\varphi}N$ (see \cite{Fernandez-Garcia-Kupeli-Unal}), for a certain function $\varphi$ on $I$. 
	Since $(M,g)$ has harmonic Weyl tensor, then $N$ is Einstein (see, for example, \cite{gebarowski}). Since $N$ is Einstein and $3$-dimensional, $N$ has constant sectional curvature and $(M,g)$ is locally conformally flat (see \cite{Some Remarks on LCF static}).
\end{proof}

\begin{remark}\rm
	Note that for $\dim M\geq 4$ and $\mu\neq -\frac{1}{\dim M-2}$ the arguments given in Lemma~\ref{lemma:W=0,nabla f eigenvector} and Theorem~\ref{Th:non-isotropic} work through and show that 
	\begin{quote}\it
		if $(M,g,f,\mu)$ is a non-isotropic qE structure with harmonic Weyl tensor and $W(\cdot,\cdot,\cdot,\nabla f)=0$, then $(M,g)$ decomposes as a warped product $I\times_\varphi N$, where $N$ is Einstein.
	\end{quote}
\end{remark}

\subsection{Isotropic quasi-Einstein manifolds}\label{subsection:isotropic}
In contrast with the non-isotropic case, if $\|\nabla f\|=0$ the level hypersurfaces of the potential function are degenerate hypersurfaces. This fact has immediate consequences on the geometry and the potential function of the qE manifold. 
\begin{lemma}\label{lemma:isotropic-first-results}
	Let $(M,g,f,\mu)$ be isotropic qE with $\mu\neq -\frac{1}{2}$. If $\operatorname{div} W=0$ and $W(\cdot,\nabla f,\cdot,\nabla f)=0$, then $\Ric(\nabla f)=\lambda \nabla f$, $\tau=4\lambda$, $\Delta f=0$ and $\nabla\lambda=-\lambda\nabla f$.
\end{lemma}
\proof	Since $\|\nabla f\|=0$, we have $\operatorname{Hes}_f(X,\nabla f)= \frac12Xg(\nabla f,\nabla f)=0$, so $\nabla_{\nabla f}\nabla f=0$. From Equation~\eqref{eq:general quasi-Einstein} we get $\Ric(\nabla f)=\lambda \nabla f$.

We evaluate Equation~\eqref{eq:weylandcotton} with $Y=\nabla f$ to see that
\[
0=(\tau-4\lambda) df(X)df(Z).
\] 
Therefore, we conclude that $\tau=4\lambda$. Now, from Equation~\eqref{eq:trace-qE} we get $\Delta f=0$. Hence, using Equation~\eqref{eq:5}, we obtain  $\nabla\lambda=-\lambda\nabla f$. 
\qed

%
%

The following result shows that the null vector field $\nabla f$ generates a parallel distribution, so the underlying manifold of the qE structure is a  Brinkmann space.

\begin{lemma}\label{lemma:LCFlorentzAQE then Walker}
Let $(M,g,f,\mu)$ be isotropic qE with $\mu\neq -\frac{1}{2}$. If $\operatorname{div} W=0$ and $W(\cdot,\cdot,\cdot,\nabla f)=0$, then $\mathcal{D}=\spanned\{\nabla f\}$ is a null parallel distribution.
\end{lemma}
\begin{proof}	
Since $\operatorname{div}W=0$  and $W(\cdot,\cdot,\cdot,\nabla f)=0$, we use Equation~\eqref{eq:weylandcotton} to see that, whenever $\mu\neq -\frac{1}{2}$, 
\begin{equation}\label{eq:lemma W=0 mu neq1}
\begin{array}{rcl}
0&=&\tau\{df(X)g(Y,Z)-df(Y) g(X,Z)\}\\
\noalign{\medskip}
& &+3\{ \rho(X,Z)df(Y)-\rho(Y,Z)df(X)\}\\
\noalign{\medskip}
&&+  g(X,Z)\rho(Y,\nabla f)
-g(Y,Z)\rho(X,\nabla f).
\end{array}
\end{equation}
Since $\|\nabla f\|=0$ and $\nabla f\neq 0$, we consider a local frame
$\mathcal{B}=\{\nabla f,U,X_1,X_2\}$ such that
the only non-zero metric products in $\mathcal{B}$ are $g(U,\nabla f)= g(X_i,X_i)=1$, for $i=1,2$. Moreover, from Lemma~\ref{lemma:isotropic-first-results}, $\Ric(\nabla f)=\lambda \nabla f$. Hence, we use Equation~\eqref{eq:lemma W=0 mu neq1} to obtain that $\rho(X_i,X_j)=0$ if $i\neq j$, $\rho(X_i,X_i)=\frac{\tau-\lambda}{3}=\lambda$ and $\rho(X_i,U)=0$.
Then, we use Equation (\ref{eq:general quasi-Einstein}) to see that the Hessian operator vanishes for every element of $\mathcal{B}$ but $U$, so  $\nabla_{X}\nabla f\propto \nabla f$ for all $X$ and, hence, $\operatorname{span}\{\nabla f\}$ is a null parallel distribution.
\end{proof}

Now, we concentrate on the analysis of the curvature components. Because $\mathcal{D}=\operatorname{span}\{\nabla f\}$ is parallel and lightlike, there are several terms of the curvature tensor that vanish identically, one has
\begin{equation}\label{eq:curvature0-terms_null-parallel-distribution}
R(\mathcal{D},\mathcal{D}^\perp,\cdot,\cdot)=0\,\text{ and }\, R(\mathcal{D}^\perp,\mathcal{D}^\perp,\mathcal{D},\cdot)=0.
\end{equation}
Note from Equation~\eqref{eq:7} that, if the Weyl tensor is harmonic, the curvature components $R(\cdot,\cdot,\cdot,\nabla f)$ can be written as follows:
\[
\begin{array}{l}
R(X,Y,Z, \nabla f)= d\lambda(X)g(Y,Z)- d\lambda(Y)g(X,Z) 
-\frac{1}{6}(X(\tau) g(Y,Z)-Y(\tau) g(X,Z))\\
\noalign{\medskip}\qquad\qquad\qquad\qquad\qquad\qquad +\mu\left\{df(Y)\Hes_f(X,Z)-df(X)\Hes_f(Y,Z)\right\}.
\end{array}
\]
Moreover, using that $\tau=4\lambda$ and $\nabla\lambda=-\lambda \nabla f$ this expression reduces to:
\begin{equation}\label{eq:curvature-nablaf}
\begin{array}{rcl}
R(X,Y,Z, \nabla f)&=& -\frac{\lambda}{3}\{df(X)g(Y,Z)- df(Y)g(X,Z)\} \\
\noalign{\medskip}&&+\mu\left\{df(Y)\Hes_f(X,Z)-df(X)\Hes_f(Y,Z)\right\}\\
\noalign{\medskip}
&=& -\left(\frac{1}{3}+\mu\right)\lambda\{df(X)g(Y,Z)- df(Y)g(X,Z)\} \\
\noalign{\medskip}&&+\mu\left\{df(X)\rho(Y,Z)-df(Y)\rho(X,Z)\right\}.
\end{array}
\end{equation}


Recall from \cite{leistner-nurowski} that a spacetime is said to be of {\it pure radiation} if $\rho(\cdot,\cdot)=\phi g(X,\cdot)g(X,\cdot)$ for a null vector field $X$. Moreover, if $X$ is parallel, then one has a {\it pure radiation metric with parallel rays}.

\begin{lemma}\label{lemma:lambda=0}
Let $(M,g,f,\mu)$ be isotropic qE with $\mu\neq -\frac{1}{2}$. If $\operatorname{div} W=0$ and $W(\cdot,\cdot,\cdot,\nabla f)=0$, then $\lambda=0$ and $(M,g)$ is a {\it pure radiation metric with parallel rays}.
\end{lemma}
\begin{proof}
We use the basis $\mathcal{B}$ from the proof of Lemma~\ref{lemma:LCFlorentzAQE then Walker}. On the one hand, by  Equation~(\ref{eq:curvature-nablaf}) we see that $R(X_i,U,X_i,\nabla f)=\frac{\lambda}{3}$. On the other hand, by Equation~\eqref{eq:curvature0-terms_null-parallel-distribution}, we know that $R(\nabla f,X_i,U,X_i)=0$. Therefore $\lambda=0$. Now the only non-zero component of the Ricci tensor is $\rho(U,U)$ so $g$ is a pure radiation metric with parallel rays (see \cite{leistner-nurowski}).
\end{proof}

\begin{remark}\rm
	The arguments of Lemmas~\ref{lemma:W=0,nabla f eigenvector}, \ref{lemma:LCFlorentzAQE then Walker} and \ref{lemma:lambda=0} do not depend on the fact that the dimension of the manifold is four. Thus these arguments show that, if $\dim M\geq 4$ with $\mu\neq -\frac{1}{\dim M-2}$, the  following more general result holds:
	\begin{quote}\it 
	If $(M,g,f,\mu)$ is an isotropic qE structure, with harmonic Weyl tensor and $W(\cdot,\cdot,\cdot,\nabla f)=0$, then $(M,g)$ is a pure radiation metric with parallel rays.
	\end{quote}
\end{remark}



The following result specifies the structure of the manifold for an isotropic qE structure in dimension four.

\begin{theorem}\label{th:isotropic}
	Let $(M,g,f,\mu)$ be an isotropic qE Lorentzian structure of dimension $4$ with $\mu\neq -\frac{1}{2}$. If $(M,g)$ has harmonic Weyl tensor and $W(\cdot, \cdot,\cdot,\nabla f)=0$, then it is locally isometric to a $pp$-wave. 			
\end{theorem}
\begin{proof}
We have already seen that, under the hypotheses of Theorem~\ref{th:isotropic}, $\nabla f$ is a null recurrent vector field (Lemma~\ref{lemma:LCFlorentzAQE then Walker}) and that the Ricci tensor is isotropic (Lemma~\ref{lemma:lambda=0}). Hence, $(M,g)$ is a $pp$-wave if, moreover, $R(\mathcal{D}^\perp,\mathcal{D}^\perp,\cdot,\cdot)=0$ (see \cite{leistner}). 
We work with the basis $\mathcal{B}=\{\nabla f,U, X_1,X_2\}$ given above. Using \eqref{eq:curvature0-terms_null-parallel-distribution}, to see that $R(\mathcal{D}^\perp,\mathcal{D}^\perp,\cdot,\cdot)=0$ we only need to verify that $R(X_1, X_2) X_i=0$, for $i=1,2$. More specifically, it is enough to check  that $R(X_1,X_2,X_i,U)=0$ and that $R(X_1,X_2,X_1,X_2)=0$. We compute
\[
0=\rho(X_1,U)=R(U,X_1,\nabla f,U)+ R(X_2,X_1,X_2,U).
\]
By \eqref{eq:curvature-nablaf} we have that $R(U,X_1,\nabla f,U)=0$, so
$R(X_2,X_1,X_2,U)=0$ as well. Analogously, we get that $R(X_1,X_2,X_1,U)=0$.
Now, we compute
\[
0=\rho(X_1,X_1)=2R(U,X_1,\nabla f,X_1)+R(X_1,X_2,X_1,X_2),
\]
but $R(U,X_1,\nabla f,X_1)=0$ by \eqref{eq:curvature0-terms_null-parallel-distribution}, so $R(X_1,X_2,X_1,X_2)=0$. 
\end{proof}

\section{Isotropic Quasi-Einstein $pp$-waves}
In this section we consider $4$-dimensional $pp$-waves $(\mathbb{R}^4,g_{ppw})$ with local coordinates $(u,v,x_1,x_2)$ such that 
\begin{equation}\label{eq:pp-wave}
g_{ppw}=2dudv+H(u,x_1,x_2) du^2+dx_1^2+dx_2^2.
\end{equation}
Note that the degenerate parallel line field is $\mathcal{D}=\operatorname{span}\{\partial_{v}\}$.
To continue the analysis of qE $pp$-waves, we distinguish the case in which the manifold is locally conformally flat.
\begin{theorem}
Let $(\mathbb{R}^4,g_{ppw})$ be a locally conformally flat $pp$-wave. Let $\mu\neq -\frac12$. Then $(\mathbb{R}^4,g_{ppw},f,\mu)$ is isotropic qE if and only if there exist functions $a$, $b_1$, $b_2$ and $c$ so that  \[H(u, x_1,x_{2}) = a(u)(x_1^2+x_2^2)+
b_1(u) x_1+	b_2(u) x_2+c(u),\] and $f=f(u)$ satisfies  $f''(u)-\mu f'(u)^2-2a(u) = 0$.
\end{theorem}
\begin{proof}
During the analysis of the previous section we observed (see Lemma~\ref{lemma:lambda=0}) that under the hypotheses $\operatorname{div} W=0$ and $W(\cdot,\cdot,\cdot,\nabla f)=0$ we have that, necessarily, $\lambda=0$. Locally conformally flat $pp$-waves with $\lambda=0$ were studied in \cite{LCFLorentzianQE}, showing that they are plane waves as above with $f$ being a solution to the given equation.
\end{proof}
Henceforth we consider the parameter $\mu$ to be arbitrary.
The following lemma is a direct computation, so we omit the details of the proof.
\begin{lemma}\label{lemma:W-divW}
Let $(\mathbb{R}^4,g_{ppw})$ be a $pp$-wave, then the only possibly non-zero terms of $W$ and $\operatorname{div} W$ are, modulo symmetries, 
\begin{equation}\label{eq:divW-pp-wave}
\begin{array}{rcl}
W(\partial_u,\partial_{x_1},\partial_u,\partial_{x_1})&=&-W(\partial_u,\partial_{x_2},\partial_u,\partial_{x_2})=\frac14\left(\partial_{x_2}^2 H-\partial_{x_1}^2 H\right),\\
\noalign{\medskip} W(\partial_u,\partial_{x_1},\partial_u,\partial_{x_2})&=&-\frac12\partial_{x_1}\partial_{x_2} H,\\
\noalign{\medskip}
\operatorname{div} W(\partial_u,\partial_{x_1},\partial_{u})&=&-\frac14\left(\partial_{x_1}^3 H+\partial_{x_1}\partial_{x_2}^2 H\right), \\
\noalign{\medskip}
\operatorname{div}W(\partial_u,\partial_{x_2},\partial_{u})&=&-\frac14\left(\partial_{x_1}^2\partial_{x_2} H+\partial_{x_2}^3 H\right).
\end{array}
\end{equation}
\end{lemma}

In what follows we are going to work repeatedly with the terms of the qE equation so, for a convenient notation, we define the operator $\mathcal{Q}(f):={\Hes}_f+\rho-\mu  df\otimes df-\lambda g$.
Before we characterize non-locally conformally flat qE $pp$-waves, we provide the following useful lemma.
\begin{lemma}\label{lemma:fnotdependv}
Let $(M,g)$ be a non-locally conformally flat $pp$-wave with degenerate parallel line field $\mathcal{D}$. If $(M,g)$ is qE, then the potential function satisfies $df(\mathcal{D})=0$.	
\end{lemma}
\begin{proof}
We use the local form given by Equation \eqref{eq:pp-wave} and begin by computing:
\begin{equation}\label{eq:f1}
\mathcal{Q}(f)_{22}= \partial_v^2 f-\mu \partial_v f.
\end{equation}
Now, the analysis is different depending on whether $\mu$ equals $0$ or not. We distinguish the two cases:

\underline{$\mu=0$:} from \eqref{eq:f1}, $f$ has the form $f(u,v,x_1,x_2)=\psi(u,x_1,x_2) v+\eta(u,x_1,x_2)$. Now, we compute
\[
\mathcal{Q}(f)_{23}=\partial_{x_1} \psi,\qquad \mathcal{Q}(f)_{24}=\partial_{x_2} \psi,
\]
from where $\psi=\psi(u)$. From $\mathcal{Q}(f)_{12}= -\lambda+\psi'$ we obtain that  $\lambda=\psi'(u)$, so computing 
\[
\mathcal{Q}(f)_{33}= -\psi'+\partial_{x_1}^2 \eta,\quad \mathcal{Q}(f)_{44}= -\psi'+\partial_{x_2}^2 \eta,\quad
\mathcal{Q}(f)_{34}=\partial_{x_1}\partial_{x_2} \eta,
\]
we get that $\eta=\eta^0+ \eta^{1}(u) x_1+ \eta^{2}(u) x_2+\frac{\psi'}{2}(x_1^2+x_2^2)$. From the qE equation again,
\[
\mathcal{Q}(f)_{13}=\eta^{1}{}^\prime+x_1 \psi''-\frac{\psi \partial_{x_1}H}2,\qquad \mathcal{Q}(f)_{14}=\eta^{2}{}^\prime+x_2 \psi''-\frac{\psi \partial_{x_2}H}2.
\]
Differentiating these expressions with respect to $x_1$ and $x_2$ we see that $\psi \partial_{x_1}^2 H=\psi \partial_{x_2}^2 H$ and that $\psi\partial_{x_1}\partial_{x_2}H=0$. If $\partial_{x_1}^2 H= \partial_{x_2}^2 H$ and $\partial_{x_1}\partial_{x_2}H=0$, from \eqref{eq:divW-pp-wave}, we obtain $W=0$, contrary to our assumption. Therefore, we conclude that $\psi=0$, so $f$ does not depend on $v$.

\underline{$\mu\neq 0$:} we argue by contradiction and assume that $f$ depends on $v$. Hence, from \eqref{eq:f1} we get that $f$ has the following form: 
\[
f(u,v,x_1,x_2)=\nu(u,x_1,x_2)-\frac{1}{\mu} \log(\mu v+\phi(u,x_1,x_2)).
\]
A new computation shows that
\[
\mathcal{Q}(f)_{23}=\frac{\mu\, \partial_{x_1} \nu}{\mu v+\phi},\quad \mathcal{Q}(f)_{24}=\frac{\mu\, \partial_{x_2} \nu}{ \mu v+\phi}.
\]
Hence $\nu=\nu(u)$ and $f(u,v,x_1,x_2)=\nu(u)-\frac{1}{\mu} \log(\mu v+\phi(u,x_1,x_2))$, so
\[
\mathcal{Q}(f)_{34}=-\frac{\partial_{x_1}\partial_{x_2} \phi}{\mu(\mu v+\phi)},
\]
and $\phi(u,x_1,x_2)=\phi^1(u,x_1)+\phi^2(u,x_2)$. We continue by computing
\[
\mathcal{Q}(f)_{13}=\frac{ 2 \mu \nu' \partial_{x_1} \phi^1- 2\partial_u\partial_{x_1} \phi^1+\mu \partial_{x_1} H}{ 2\mu(\mu v+ \phi)},
\]
and, since the numerator must vanish identically, we differentiate with respect to $x_2$ to see that $\mu\partial_{x_1}\partial_{x_2}H=0$. So $H(u,x_1,x_2)=H^1(u,x_1)+H^2(u,x_2)$. We compute
\[
\mathcal{Q}(f)_{12}= -\lambda+\frac{\mu \nu'}{\mu v+\phi},\,\mathcal{Q}(f)_{33}=-\lambda-\frac{\partial_{x_1}^2\phi^1}{\mu(\mu v+\phi)},\,\mathcal{Q}(f)_{44}=-\lambda- \frac{\partial_{x_2}^2\phi^2}{\mu(\mu v+\phi)}.
\]
Hence, comparing these expressions we conclude that $\partial_{x_1}^2\phi^1=\partial_{x_2}^2\phi^2= -\mu^2 \nu'$. So $\partial_{x_1}^3 \phi^1=\partial_{x_2}^3\phi^2=0$ and $\phi(u,x_1,x_2)=-\frac{\mu^2 v'}{2} (x_1^2+x_2^2)+\phi^{10}(u)x^1+\phi^{20}(u)x^2+\phi^0(u)$. Also, we get that $\lambda(u,v,x_1,x_2)=\frac{\mu \nu'}{\mu v+\phi}$.

Now, we check that 
\[
\begin{array}{rcl}
\mathcal{Q}(f)_{13}&=&\frac{-2 \phi^{10}{}^\prime+\mu(2\phi^{10} \nu^\prime+2x_1\mu(-\mu (\nu^\prime)^2+\nu^{\prime \prime})+\partial_{x_1} H^1)}{2\mu(\mu v+\phi)},\\
\noalign{\medskip}
\mathcal{Q}(f)_{14}&=&\frac{-2 \phi^{20}{}^\prime+\mu(2 \phi^{20} \nu'+2 x_2 \mu(-\mu (\nu')^2+\nu^{\prime \prime})+\partial_{x_2}H^2)}{2\mu(\mu v+\phi)}.
\end{array}
\]
Since both numerators must vanish identically, we differentiate the first one with respect to $x_1$  and the second one with respect to $x_2$ to see that
\[
\partial_{x_1}^2 H^1=-2\mu(-\mu (\nu^\prime)^2+\nu^{\prime \prime}),\qquad \partial_{x_2}^2 H^2=-2 \mu(-\mu (\nu')^2+\nu^{\prime \prime}).
\]
Hence $\partial_{x_1}^2 H= \partial_{x_2}^2 H$ and, since we already know that $\partial_{x_1}\partial_{x_2}H=0$, we have from \eqref{eq:divW-pp-wave} that $W=0$. This contradicts the assumption that $(\mathbb{R}^4,g_{ppw})$ is non-locally conformally flat, so we conclude that $f$ does not depend on $v$.
\end{proof}

The following result shows that a necessary  condition for a $pp$-wave which is not locally conformally flat to be isotropic qE is precisely that the Weyl tensor is harmonic. And, moreover, every $pp$-wave with harmonic Weyl tensor is, at least locally, an isotropic qE manifold for any $\mu$.
\begin{theorem}\label{th:isotropic-pp-wave}
	Let $(U\subset\mathbb{R}^4,g_{ppw})$ be a non-locally conformally flat $pp$-wave. The following statements are equivalent: 
	\begin{enumerate}
		\item[(i)] $(U,g_{ppw})$ is isotropic qE,
		\item[(ii)] $W$ is harmonic,		
		\item[(iii)] $\Delta_x H=\varphi(u)$, where $\Delta_x H= \partial_{x_1}^2 H+
		\partial_{x_2}^2H$.
	\end{enumerate}
	If any of these conditions holds, then $W(\cdot,\cdot,\cdot,\nabla f)=0$ and $f=f(u)$ is given by
	\begin{equation}\label{eq:qE-final}
	f''(u)-\mu f'(u)^2-\frac{1}{2}\Delta_x H=0.
	\end{equation}
\end{theorem}
\begin{proof}
We consider a non-locally conformally flat $pp$-wave $(\mathbb{R}^4,g_{ppw})$ as given in \eqref{eq:pp-wave} and work in local coordinates. We assume $(\mathbb{R}^4,g_{ppw},f,\mu)$ is qE. A direct computation shows that for arbitrary $f$ we have
\begin{equation}\label{eq:gradiente}
\|\nabla f\|^2=2\partial_u f \partial_v f-H \left(\partial_v f\right)^2+\left(\partial_{x_1} f\right)^2+\left(\partial_{x_2} f\right)^2.
\end{equation}
By Lemma~\ref{lemma:fnotdependv} $f$ does not depend on $v$, so
$\|\nabla f\|^2=\left(\partial_{x_1} f\right)^2+\left(\partial_{x_2} f\right)^2$. Hence $\partial_{x_1} f=\partial_{x_2}f=0$ and $f=f(u)$. A direct computation of the qE equation shows that
\[
\mathcal{Q}(f)_{33}=\mathcal{Q}(f)_{44}= -\lambda,
\]
so $\lambda=0$. Now, the only non-vanishing term in the qE equation is
\begin{equation}\label{eq:qE-isotropic-lasteq}
\mathcal{Q}(f)_{11}=f''-\mu (f')^2-\frac12\left(\partial_{x_1}^2 H+\partial_{x_2}^2H\right).
\end{equation}
So $\Delta_x H= \partial_{x_1}^2 H+
\partial_{x_2}^2H=2(f''-\mu (f')^2)$ and (i) implies (iii). Now, assuming 
$\Delta_x H=\varphi(u)$, we differentiate with respect to $x_1$ and to $x_2$ to see that $\partial_{x_1}^3 H+\partial_{x_1}\partial_{x_2}^2 H=0$ and $\partial_{x_1}^2\partial_{x_2} H+\partial_{x_2}^3 H=0$. Hence, from Equation~\eqref{eq:divW-pp-wave} we conclude that $\operatorname{div}W=0$ and (ii) follows, so (iii) implies (ii).

Finally, we assume (ii) holds. Then, from \eqref{eq:divW-pp-wave} we have $\partial_{x_1}^3 H+\partial_{x_1}\partial_{x_2}^2 H=0$ and $\partial_{x_1}^2\partial_{x_2} H+\partial_{x_2}^3 H=0$. So $\partial_{x_1}^2 H+\partial_{x_2}^2 H$ is a function of $u$ and Equation~\eqref{eq:qE-isotropic-lasteq} admits a solution $f=f(u)$ in an open subset $U$. So $(U,g_{ppw},f, \mu)$ is isotropic qE and (i) follows.  

If (i) holds, then $f=f(u)$ satisfies \eqref{eq:qE-final} and, moreover, $\nabla f=  f'\partial_v$. Hence, from the expression of $W$ in \eqref{eq:divW-pp-wave} we have that $W(\cdot,\cdot,\cdot,\nabla f)=0$ and the theorem follows.
\end{proof}

Note that in Theorem~\ref{th:isotropic-pp-wave} there is no restriction on $\mu$. Thus, in particular, it works for $\mu=-\frac12$. Hence a $4$-dimensional $pp$-wave with harmonic Weyl tensor is conformally Einstein.

\begin{corollary}
A $4$-dimensional $pp$-wave is isotropically conformally Einstein if and only if $\operatorname{div} W=0$.
\end{corollary}

\end{document}